\documentclass{amsart}
\usepackage{amsmath,graphicx,latexsym}

\newtheorem{thm}{Theorem}
\newtheorem{lem}[thm]{Lemma}

\theoremstyle{definition}
\newtheorem{defn}[thm]{Definition}
\newtheorem{rmk}[thm]{Remark}

\newcommand{\CPb}{\overline{\mathbb{CP}}{}^{2}}
\newcommand{\CP}{{\mathbb{CP}}{}^{2}}
\newcommand{\R}{\mathbb{R}}
\newcommand{\Z}{\mathbb{Z}}
\newcommand{\M}{M_{k,n}^{p,r}}
\newcommand{\sM}{M_{k,1}^{p,r}}
\newcommand{\ssM}{M_{k,1}^{1,1}}
\newcommand{\Mkn}{M_{k,n}^{1,1}}
\newcommand{\Mknm}{M_{k,n}^{1,1}(m)}

\title[Exotic 4-manifolds with zero signature]
{Exotic smooth structures on \\
4-manifolds with zero signature}

\author[A. Akhmedov]{Anar Akhmedov}
\address{School of Mathematics,
University of Minnesota, 
Minneapolis, MN, 55455, USA}
\email{akhmedov@math.umn.edu}

\author[B. D. Park]{B. Doug Park}
\address{Department of Pure Mathematics, 
University of Waterloo,
Waterloo, ON, N2L 3G1, Canada}
\email{bdpark@math.uwaterloo.ca}

\date{March 26, 2010.  Revised on June 3, 2010}


\begin{document}

\begin{abstract}
For every integer $k\geq 2$, we construct infinite families of mutually 
nondiffeomorphic irreducible smooth structures on the topological $4$-manifolds 
$(2k-1)(S^2\times S^2)$ and $(2k-1)(\CP\#\CPb)$, 
the connected sums of $2k-1$ copies of $S^2\times S^2$ and $\CP\#\CPb$.  
\end{abstract}

\maketitle

\section{Introduction}

Let $M$\/ denote a closed smooth $4$-manifold.  
Given an integer $k\geq 1$, let $kM$\/ denote the connected sum of $k$\/ copies of $M$.  
Let $\CP$ denote the complex projective plane and let $\CPb$ denote the underlying 
smooth $4$-manifold $\CP$ equipped with the opposite orientation.  
Let $\CP\#\CPb$ denote the connected sum of $\CP$ and $\CPb$.  
To state our results, it will be convenient to introduce the following terminology.  

\begin{defn}\label{defn: infty2}
We say that a $4$-manifold $M$\/ has\/ $\infty^2$-\emph{property}\/ if there exist infinitely many mutually nondiffeomorphic irreducible symplectic $4$-manifolds and infinitely many mutually nondiffeomorphic irreducible nonsymplectic $4$-manifolds, all of which are homeomorphic to $M$.  
\end{defn}

The main goal of this paper is to prove the following.  

\begin{thm}\label{thm: main}
Both\/ $(2k-1)(S^2\times S^2)$ and $(2k-1)(\CP\#\CPb)$ have $\infty^2$-property  
for every integer $k\geq 2$.
\end{thm}

It was already proved in \cite{AP: spin, AP: nonspin} that $(2k-1)(S^2\times S^2)$ 
and $(2k-1)(\CP\#\CPb)$ have $\infty^2$-property  
when $k\geq 138$ and $k\geq 25$, respectively.  
At the moment, 
it is unknown to the authors whether there is any overlap between the $4$-manifolds constructed in 
\cite{AP: spin, AP: nonspin} and in this paper.  
The proof of Theorem~\ref{thm: main} is given in 
Sections~\ref{sec: model}--\ref{sec: (2k-1)(PPb)}.  
Our strategy is to apply the `reverse engineering' technique of \cite{FPS} to 
a suitably chosen nontrivial genus 2 surface bundle over a genus $k+1$ surface.

\section{Model complex surfaces}
\label{sec: model}

Let $\Sigma_g$ denote a closed genus $g$\/ Riemann surface.  
For every integer $k\geq 1$, 
there is a free $\Z/2$ action on a genus $2k+1$ surface, 
$\tau_{k+1}:\Sigma_{2k+1}\rightarrow\Sigma_{2k+1}$, 
by rotating 180 degrees about the `middle' $(k+1)$-th hole, 
such that $\Sigma_{2k+1}/\langle\tau_{k+1}\rangle=\Sigma_{k+1}$.  
When $k=1$, we just get $\tau_2$ in \cite{AP: S2xS2}.  
Let $\tau_1$ be the elliptic involution on $\Sigma_2$ in 
Section~2 of \cite{AP: S2xS2}.  
There exists a free $\Z/2$ action $(\tau_1,\tau_{k+1})$ on the product $\Sigma_2\times\Sigma_{2k+1}$ for every integer $k\geq 1$.  
Let $X_k$ denote the quotient space and 
let $q:\Sigma_2\times\Sigma_{2k+1}\rightarrow X_k$ denote the quotient map.  
$X_k$ is a minimal complex surface of general type (cf.~\cite{lopes-pardini}), 
which is the total space of a nontrivial genus 2 surface bundle over a genus $k+1$ surface.  
We have $e(X_k)=e(\Sigma_2\times\Sigma_{2k+1})/2=(-2)(-4k)/2=4k$, $\sigma(X_k)=0$, 
$b_1(X_k)=2k+4$ and $b_2(X_k)=8k+6$.  

Let $\{a_1,b_1,a_2,b_2\}$ and $\{c_1,d_1,\dots,c_{2k+1},d_{2k+1}\}$ be the 
set of simple closed curves representing the standard generators of 
$\pi_1(\Sigma_2,z_0)$ and $\pi_1(\Sigma_{2k+1},w_0)$, respectively.  
For the sake of brevity, we will usually abuse notation and write 
$a_1=q(a_1\times\{w_0\})$ and $c_1=q(\{z_0\}\times c_1)$, etc.  
The next lemma is the analogue of Lemmas 5 and 6 in \cite{AP: S2xS2}.  

\begin{lem}\label{lem: gen Xk}
The following\/ $2k+4$ loops generate $\pi_1(X_k):$
\begin{equation*}
a_1,b_1,c_1,d_1,\dots,c_k,d_k,
\tilde{c}_{k+1},d_{k+1}.  
\end{equation*}
Moreover, the loops $c_1,d_1,\dots,c_k,d_k,
\tilde{c}_{k+1},d_{k+1}$ represent elements of infinite order
in $\pi_1(X_k)$.
\end{lem}

\begin{proof}
As in \cite{AP: S2xS2}, we can verify that 
\begin{align*}
a_2=\tilde{c}_{k+1}^{-1} a_1\tilde{c}_{k+1}, & \ \ 
b_2=\tilde{c}_{k+1}^{-1} b_1 \tilde{c}_{k+1}, \\
c_{2k+2-j}=\tilde{c}_{k+1}^{-1} c_j \tilde{c}_{k+1}, & \ \ 
d_{2k+2-j}=\tilde{c}_{k+1}^{-1} d_j \tilde{c}_{k+1}, 
\end{align*}
where $j=1,2,\dots,k$.  
The bundle projection map
\begin{equation*}
X_k = \frac{\Sigma_2\times \Sigma_{2k+1}}{\Z/2} \longrightarrow 
\frac{\Sigma_{2k+1}}{\Z/2} = \Sigma_{k+1}
\end{equation*}
maps $c_1,d_1,\dots,c_k,d_k,\tilde{c}_{k+1},d_{k+1}$ 
to the standard generators of $\pi_1(\Sigma_{k+1})$.  
\end{proof}

The intersection form of $X_k$ is given by $(4k+3)H$, where 
\begin{equation}\label{eq: H}
H = \left[\begin{array}{cc}
0&1 \\
1&0
\end{array} 
\right].
\end{equation} 
A basis for the intersection form of $X_k$ is given by 
the following $4k+3$ geometrically dual pairs:   
\begin{equation*}
\begin{array}{cc}
([a_1 \times c_1], -[b_1 \times d_1]), & ([a_1 \times d_1], [b_1 \times c_1]),\\
([a_2 \times c_1], -[b_2 \times d_1]), & ([a_2 \times d_1], [b_2 \times c_1]),\\
\vdots & \vdots \\
([a_1 \times c_k], -[b_1 \times d_k]), & ([a_1 \times d_k], [b_1 \times c_k]),\\
([a_2 \times c_k], -[b_2 \times d_k]), & ([a_2 \times d_k], [b_2 \times c_k]),\\[2pt]
([(\tilde{a}_1\tilde{a}_2) \times \tilde{c}_{k+1}], -[b_1 \times d_{k+1}]), &
([a_1 \times d_{k+1}], [(\tilde{b}_1\tilde{b}_2) \times \tilde{c}_{k+1}]), \\[2pt]
([\Sigma_2 \times \{w_0\}], [\{z_0\} \times \Sigma_{2k+1}]). &
\end{array} 
\end{equation*}
Here, $[\,\cdot\,]$ denotes the homology class of the image $q(\,\cdot\,)$ in the quotient manifold $X_k$ for short.  
All throughout, the $\tilde{c}_{k+1}$ path plays the role of $\tilde{c}_2$ in \cite{AP: S2xS2}.

\section{Exotic $(2k-1)(S^2\times S^2)$}
\label{sec: (2k-1)(S2xS2)}

By performing appropriate $2k+4$ torus surgeries on $X_k$, we obtain exotic smooth structures on $(2k-1)(S^2\times S^2)$.  
For example, we can choose to perform the following torus surgeries:    
\begin{equation}\label{eq: 2k+4 surgeries}
\begin{array}{cc}
(a_1' \times c_1', a_1', -1), & 
(b_1' \times c_1'', b_1', -1), \\
(a_2' \times c_1', c_1', -1), &
(a_2'' \times d_1', d_1', -n), \\
(a_2' \times c_2', c_2', -1/p), &
(a_2'' \times d_2', d_2', -1/r), \\
(a_2' \times c_3', c_3', -1), &
(a_2'' \times d_3', d_3', -1), \\
 \vdots & \vdots
\\
(a_2' \times c_k', c_k', -1), &
(a_2'' \times d_k', d_k', -1), \\
(b_1'' \times d_{k+1}', d_{k+1}', -1), &
((\tilde{b}_1\tilde{b}_2) \times \tilde{c}_{k+1}', \tilde{c}_{k+1}', -1), 
\end{array} 
\end{equation}
where $k,n,p,r$\/ are integers satisfying
\begin{equation}\label{eq: npr}
k\geq 1, \ \
n\geq 1, \ \  
p\geq 0 \ \ {\rm and} \ \ 
r\geq 0.  
\end{equation}
The prime and double prime notations are explained in \cite{FPS}.  
The fourth, the fifth and the sixth surgeries are Luttinger surgeries 
(cf.~\cite{luttinger, ADK})
when $n=1$, $p\geq 1$ and $r\geq 1$, respectively.  

Let $\M$ denote the resulting closed $4$-manifold.  
When $n=1$, $\sM$ is symplectic for every triple $k,p,q$\/ satisfying (\ref{eq: npr}).  
When $k=1$, there are no $-1/p$\/ and $-1/r$\/ surgeries, and 
we just get $M_n^1$ in \cite{AP: S2xS2}.   
Using exactly the same argument as in \cite{AP: S2xS2}, 
we can prove the following.  

\begin{lem}\label{lem: presentation}
$\pi_1(\M)$ is generated by 
\begin{equation}\label{eq: generators}
a_1,b_1,a_2,b_2,c_1,d_1,\dots,c_k,d_k,
\tilde{c}_{k+1},d_{k+1}.  
\end{equation}
If $k\geq 2$, then the following relations hold in $\pi_1(\M):$  
\begin{gather*}
a_2=\tilde{c}_{k+1}^{-1} a_1\tilde{c}_{k+1}, \ \  
b_2=\tilde{c}_{k+1}^{-1} b_1 \tilde{c}_{k+1}, \ \ 
b_1=\tilde{c}_{k+1}^{-1} b_2 \tilde{c}_{k+1}, \\
[b_2,d_{k+1}]=1, \ \ 
[a_1^{-1}b_1^{-1}a_2,d_{k+1}]=1, \ \
[a_2^{-1}b_2^{-1}a_1,d_{k+1}]=1, \\
[b_1^{-1}, d_1^{-1}]=a_1,\ \  [a_1^{-1}, d_1]=b_1,\ \ 
[b_2^{-1}, d_1^{-1}]=c_1,\ \ [b_2, c_1^{-1}]^n =d_1,\\
[b_2^{-1}, d_2^{-1}]=c_2^p,\ \ [b_2, c_2^{-1}]=d_2^r,
\ \ \dots\, , \ \ 
[b_2^{-1}, d_k^{-1}]=c_k,\ \ [b_2, c_k^{-1}]=d_k,\\
\tilde{c}_{k+1}^{-1} a_1 a_2 \tilde{c}_{k+1} a_1^{-1} a_2^{-1}=d_{k+1}, \ \
[a_1, d_{k+1}^{-1}]=\tilde{c}_{k+1},\\
[a_1,c_1]=1, \ \ [b_1,c_1]=1,\ \ [a_2,c_1]=1, \ \ [a_2,d_1]=1, \ \ 
[b_1,d_{k+1}]=1, \\
[a_2,c_2]=1, \ \ [a_2,d_2]=1, \ \ \dots\, , \ \  
[a_2,c_k]=1, \ \ [a_2,d_k]=1, \\
[a_1,c_2]=1, \ \ [b_1,d_2]=1, \ \ [a_1,d_2]=1, \ \ [b_1,c_2]=1,\\
\vdots \hspace{1.9cm} \vdots \hspace{2cm} \vdots \hspace{2cm} \vdots 
\hspace{.7cm} \\
[a_1,c_k]=1, \ \ [b_1,d_k]=1, \ \ [a_1,d_k]=1, \ \ [b_1,c_k]=1,\\
[a_1,b_1][a_2,b_2]=1, \ \
[c_1,d_1]\cdots[c_k,d_k][\tilde{c}_{k+1},d_{k+1}]=1.
\end{gather*}
\end{lem}

\begin{lem}\label{lem: pi_1}
If $k\geq 2$, then $\pi_1(\M) \cong \Z/p \oplus \Z/r$.  In particular, 
$\pi_1(\Mkn)=0$
for every pair of integers $k\geq 2$ and $n\geq 1$.
\end{lem}

\begin{proof}
By arguing exactly the same way as in the proof of Theorem~9 in \cite{AP: S2xS2}, 
with $\tilde{c}_{k+1}$ and $d_{k+1}$ playing the roles of $\tilde{c}_2$ and $d_2$ in \cite{AP: S2xS2}, respectively, 
we can show that $[b_1,b_2]=1$ and then $a_1=1$.  
From $a_1=1$, we can easily deduce that all other generators are trivial except for $c_2$ and $d_2$.  Since 
\begin{equation*}
[c_2,d_2]=[c_1,d_1]\cdots[c_k,d_k][\tilde{c}_{k+1},d_{k+1}]=1,
\end{equation*} 
the remaining generators $c_2$ and $d_2$ commute.  
From Lemma~\ref{lem: gen Xk}, we deduce that $c_2$ and $d_2$ have infinite order before the surgeries (\ref{eq: 2k+4 surgeries}).  
We can now conclude that $\pi_1(\M)$ is abelian and is isomorphic to 
$\Z/p \oplus \Z/r$.  
\end{proof}

For the remainder of this section, let $p=r=1$ and $k\geq 2$.
The intersection form of $\Mkn$ is isomorphic to 
$(2k-1)H$\/ (see (\ref{eq: H})) with a basis given by
\begin{equation}\label{eq: Mkn basis}
\begin{array}{cc}
([a_1 \times c_2], -[b_1 \times d_2]), & ([a_1 \times d_2], [b_1 \times c_2]),\\
\vdots & \vdots \\
([a_1 \times c_k], -[b_1 \times d_k]), & ([a_1 \times d_k], [b_1 \times c_k]),\\[2pt]
([\Sigma_2 \times \{w_0\}], [\{z_0\} \times \Sigma_{2k+1}]). &
\end{array} 
\end{equation}
Hence the simply connected $4$-manifolds $\{ \Mkn \mid n\geq 1\}$ are all homeomorphic to $(2k-1)(S^2\times S^2)$ 
by Freedman's theorem in \cite{freedman}.  

Let $Z_k$ denote the spin symplectic $4$-manifold obtained from $X_k$ by performing
$2k+3$ Luttinger surgeries in (\ref{eq: 2k+4 surgeries}) with $p=r=1$, 
but not $(a_2'' \times d_1', d_1', -n)$ surgery.  
In other words, $\Mkn$  is obtained from $Z_k$ by performing $(a_2'' \times d_1', d_1', -n)$ surgery.  
We have $e(Z_k)=4k$, $\sigma(Z_k)=0$, $b_1(Z_k)=1$, $b_2(Z_k)=4k$, 
and $b_2^+(Z_k)=2k$.  
Let $A$\/ and $B$\/ denote the 2-dimensional cohomology classes of 
either $Z_k$ or $\Mkn$ 
that are Poincar\'e dual to the homology classes of $q(\Sigma_2 \times \{w_0\})$ and $q(\{z_0\} \times \Sigma_{2k+1})$, respectively.  

If $SW_{Z_k}(L)\neq 0$, then by the adjunction inequality, 
$L = sA + tB$, where $s$\/ and $t$\/ are even integers satisfying
$|s|\leq 2k$ and $|t|\leq 2$.  
Since the dimension of the Seiberg-Witten moduli space for $L$\/ has to be
nonnegative,  
\begin{equation*}
L^2 = 2st \geq 2e(Z_k)+ 3\sigma(Z_k) = 8k.  
\end{equation*}
Hence $s=\pm 2k$, $t=\pm 2$, $L=\pm(2kA + 2B)=\mp c_1(Z_k)$, and
by Taubes's theorem in \cite{taubes},
\begin{equation*}
|SW_{Z_k}(\pm(2kA+2B))|=1.
\end{equation*}
It now follows from \cite{FPS, ABP} that 
$SW_{\Mkn}(L)\neq 0$ only when $L=\pm(2kA+2B)$, and 
\begin{equation*}
|SW_{\Mkn}(\pm(2kA+2B))|=n.
\end{equation*} 
We conclude that $\Mkn$'s are mutually nondiffeomorphic.  
By Taubes's theorem in \cite{taubes},  
$\Mkn$ is nonsymplectic if $n\geq 2$.

\section{Exotic $(2k-1)(\CP\#\CPb)$}
\label{sec: (2k-1)(PPb)} 

Throughout this section, let $k\geq 2$ be an integer.  
Given an integer $m\geq 2$, 
let $\Mknm$ denote the result of performing an $m$-surgery in $\Mkn$ 
along one of the tori in (\ref{eq: Mkn basis}), say 
\begin{equation}\label{eq: log transform}
(a_1'\times c_2', c_2', +m).
\end{equation}  
By convention, we define $\Mkn(1)=\Mkn$.  

\begin{lem}\label{lem: complement}
If $k\geq 2$, then $\pi_1(\Mkn\setminus q(a_1'\times c_2'))=0$.  
\end{lem}

\begin{proof}
The $q(a_1'\times c_2')$ torus intersects the $q(b_1\times d_2)$ torus 
negatively once in $\Mkn$.  
Hence any meridian of the $q(a_1'\times c_2')$ torus will be a conjugate of 
$[b_1,d_2]^{\pm 1}$.  
It follows that $\pi_1(\Mkn\setminus q(a_1'\times c_2'))$ 
is normally generated by the generators listed in (\ref{eq: generators}).  
Note that all the relations in Lemma~\ref{lem: presentation} continue to hold in
$\pi_1(\Mkn\setminus q(a_1'\times c_2'))$, except possibly for the relation 
$[b_1,d_2]=1$.  
In the proof of Lemma~\ref{lem: pi_1} with $p=r=1$, we were able to kill all generators in (\ref{eq: generators}) 
without making use of the relation $[b_1,d_2]=1$.  
Hence the generators in (\ref{eq: generators}) are still trivial in 
$\pi_1(\Mkn\setminus q(a_1'\times c_2'))$ and our lemma follows.  
\end{proof}

Since $\pi_1(\Mkn\setminus q(a_1'\times c_2'))=0$, 
we can interpret (\ref{eq: log transform}) as a `generalized logarithmic 
transformation' of multiplicity $m$.   
When $n=1$, we can perturb the symplectic form on $\ssM$ such that
$q(a_1'\times c_2')$ becomes a symplectic submanifold.  
Hence our generalized logarithmic transformation can be performed 
symplectically (cf.~\cite{symington}), and 
the resulting $4$-manifold $\ssM (m)$ is symplectic.  
Note that $\pi_1(\Mknm)=0$, $e(\Mknm)=4k$, 
$\sigma(\Mknm)=0$, and $b_2^+(\Mknm)=2k-1$ for every triple of integers 
$k\geq 2$, $n\geq 1$, and $m\geq 1$.  

Since $\pi_1(\Mkn\setminus q(a_1'\times c_2'))=0$, 
Corollary~21 in \cite{park} applies and we conclude that 
\begin{equation}\label{eq: sw Mknm}
SW_{\Mknm}(L) = \left\{
\begin{array}{cl}
n & \text{if \ } L = \pm(2kA+2B)+jT, \\[2pt]
0 & {\rm otherwise},  
\end{array}\right.
\end{equation}
where $j\in \{-(m-1), -(m-3), \dots, m-3, m-1 \}$,  
and $T$\/ is the cohomology class of $\Mknm$
that is Poincar\'e dual to the core torus 
of the logarithmic transformation.   

Since every Seiberg-Witten basic class is characteristic, we must have
\begin{equation*}
w_2(\Mknm)\equiv \pm(2kA+2B)+(m-1)T\equiv (m-1)T \pmod{2}.
\end{equation*}
Hence we conclude that 
\begin{equation*}
w_2(\Mknm)\equiv \left\{
\begin{array}{cl}
0 & \text{if \ $m$\/ is odd}, \\[2pt]
T & \text{if \ $m$\/ is even}.  
\end{array}\right.
\end{equation*}
Since $T$\/ is primitive, $T\not\equiv 0\pmod{2}$.  
It follows that $\Mknm$ is spin if $m$\/ is odd 
and nonspin if $m$\/ is even.
By Freedman's theorem in \cite{freedman}, $\Mknm$ is homeomorphic to
$(2k-1)(S^2\times S^2)$ if $m$\/ is odd and 
homeomorphic to $(2k-1)(\CP\#\CPb)$ if $m$\/ is even.  

If $L$\/ and $L'$ are Seiberg-Witten basic classes of $\Mknm$,
then $(L-L')^2$ is either $0$ or $32k$\/  
by (\ref{eq: sw Mknm}).  
Hence $(L-L')^2$ is never $-4$ and we can deduce that every 
$\Mknm$ is irreducible.  
This concludes the proof of Theorem~\ref{thm: main}.

\begin{rmk}
We should point out that the torus surgeries in (\ref{eq: 2k+4 surgeries}) 
and (\ref{eq: log transform}) are not the only ones we could have chosen.  
We have verified that 
many other combinations of surgeries work just as well and give rise to alternative families of exotic smooth structures.  
\end{rmk}

\begin{rmk}
Recall that $M_1^1$ in \cite{AP: S2xS2} contains genus $2$ surfaces $q(\Sigma_2 \times \{\text{pt}\})$ with self-intersection $0$.  
To obtain an alternative construction of exotic smooth structures on $(2k-1)(S^2 \times S^2)$ for $k\geq 2$,
we can fiber sum $M_1^1$ with $k-1$ copies of $\Sigma_2 \times T^2$, along genus $2$ surfaces $q(\Sigma_2 \times \{\text{pt}\})$ and $\Sigma_2 \times \{\text{pt}'\}$, and then perform $4(k-1)$ Luttinger surgeries, 
$4$ Luttinger surgeries in each copy of $\Sigma_2\times T^2$ 
(cf.~\cite{AP: odd, ABP}).  
Details of this and a few other similar constructions will appear in a later version.  
\end{rmk}

\section*{Acknowledgments}
The first author was partially supported by an NSF grant.    
The second author was partially supported by an NSERC discovery grant.  
The authors thank Robert E. Gompf for helpful discussions.

\end{document}